\documentclass[12pt,letterpaper]{article}

\usepackage{amssymb,amsfonts,amscd,amsthm}
\usepackage[all,arc]{xy}
\usepackage{enumerate, bbm}
\usepackage{mathrsfs}
\usepackage{tikz}
\usepackage[left=0.9in,top=0.9in,right=0.9in,bottom=0.9in]{geometry}
\usepackage{mathtools}
\usepackage{hyperref}
\usepackage{color}
\usepackage{graphicx}
\usepackage{enumitem} 
\graphicspath{ {TexImages/} }

\newtheorem{thm}{Theorem}[section]
\newtheorem{cor}[thm]{Corollary}
\newtheorem{prop}[thm]{Proposition}
\newtheorem{lem}[thm]{Lemma}

\newtheorem{quest}[thm]{Question}

\newtheorem{conj}[thm]{Conjecture}
\newtheorem{obs}[thm]{Observation}

\theoremstyle{definition}
\newtheorem{defn}{Definition}

\hypersetup{
	colorlinks,
	citecolor=blue,
	filecolor=blue,
	linkcolor=blue,
	urlcolor=blue,
	linktocpage
}

\setenumerate[1]{label=\thesection.\arabic*.} 
\setenumerate[2]{label*=\arabic*.} 
\setlist[enumerate]{itemsep=2ex, topsep=2ex} 
\setlist[itemize]{itemsep=2ex, topsep=2ex}

\newcommand{\al}{\alpha}

\newcommand{\sig}{\sigma}

\newcommand{\Om}{\Omega}

\newcommand{\half}{\frac{1}{2}}

\newcommand{\sm}{\setminus}
\newcommand{\sub}{\subseteq}

\renewcommand{\c}[1]{\mathcal{#1}}

\newcommand{\sgn}{\sig}

\title{The Wiener Index of Signed Graphs}
\author{Sam Spiro\footnote{Dept.\ of Mathematics, UCSD {\tt sspiro@ucsd.edu}. This material is based upon work supported by the National Science Foundation Graduate Research Fellowship under Grant No. DGE-1650112.}}
\date{\today}
\begin{document}
	\maketitle
	
\begin{abstract}
	The Wiener index of a graph $W(G)$ is a well studied topological index for graphs.  An outstanding problem of  \v{S}olt{\'e}s is to find graphs $G$ such that $W(G)=W(G-v)$ for all vertices $v\in V(G)$, with the only known example being $G=C_{11}$.  We relax this problem by defining a notion of Wiener indices for signed graphs, which we denote by $W_\sigma(G)$, and under this relaxation we construct many signed graphs such that $W_\sigma(G)=W_\sigma(G-v)$ for all $v\in V(G)$.   This ends up being related to a problem of independent interest, which asks when it is possible to $2$-color the edges of a graph $G$ such that there is a path between any two vertices of $G$ which uses each color the same number of times.


	
	
	%
\end{abstract}

\section{Introduction}
The Wiener index of a graph $G$ is defined to be $W(G)=\half \sum_{u,v\in V(G)} d_G(u,v)$, and throughout we adopt the convention that $d_G(u,v)=\infty$ if $u,v$ are in different components of $G$.  This definition was introduced by Wiener~\cite{W} in order to solve a problem in chemistry, and since then numerous results involving the Wiener index have been published, see for example \cite{DEG, KST,T,xu2014survey} and the numerous references therein.  Of particular interest to us is a problem of   \v{S}olt{\'e}s~\cite{S} which asks to find graphs $G$ which satisfy $W(G)=W(G-v)$ for all vertices $v\in V(G)$.  The only known example is $G=C_{11}$, and we refer the reader to papers of Knor, Majstrovi\'c, and \v{S}krevoski~\cite{KMS,KMS3,KMS2} for a thorough treatment of this problem.

Our main goal of this paper is to establish a notion of Wiener indices for signed graphs.  A \textit{signed graph} is a pair $(G,\sig)$ where $G$ is a graph and $\sig$ is a function from $E(G)$ to $\{\pm 1\}$ which is called a \textit{signing}.  Given a signed graph $(G,\sig)$ and a subgraph $G'\sub G$, we abuse notation slightly by writing $(G',\sig)$ to be the signed graph where $\sig$ is the restriction of the map $\sig:E(G)\to \{\pm 1\}$ to $E(G')$.   In our drawings of signed graphs we represent negative edges with dashed lines and positive edges with solid lines, see for example Figure~\ref{fig:P1}.  We say that a path $P$ is a \textit{$uv$-path} if its endpoints are $u$ and $v$.

\begin{defn}
	If $P$ is a path in $G$ and $\sig$ is a signing of $G$, we define $\sig(P):=\sum_{e\in P} \sig(e)$.  For $u,v\in V(G)$ we define the \textit{signed distance} $d_{G,\sig}(u,v)=\min_P |\sig(P)|$ where the minimum ranges over all $uv$-paths $P$, and when $G$ is understood we will simply denote this by $d_\sig(u,v)$.  We define the Wiener index $W_\sig(G)$ of the signed graph $(G,\sig)$ by $W_\sig(G)=\half \sum_{u,v\in V(G)} d_\sig(u,v)$.  
\end{defn}
Observe that if $\sig$ is a constant function, then $d_\sig(u,v)=d(u,v)$, and hence $W_\sig(G)=W(G)$.  In particular, if $W(G)=W(G-v)$ for all $v\in V(G)$, then there exists a (constant) signing $\sig$ of $G$ such that $W_\sig(G)=W_\sig(G-v)$.  Thus the problem of finding signed graphs $(G,\sig)$ with $W_\sig(G)=W_\sig(G-v)$ can be viewed as a relaxation of \v{S}olt{\'e}s' problem.  

In this paper we give many examples of signed graphs satisfying $W_\sig(G)=W_\sig(G-v)$ for all $v\in V(G)$, and even with $W_\sig(G)=W_\sig(G-S)$ for any set $S$ of size less than some value $k$.  All of these examples rely heavily on the fact that, in the signed setting, it is possible to have $W_\sig(G)=0$.  More generally, we use the following definition.

\begin{defn}
	We say that a signing $\sig$ of a graph $G$ is \textit{$k$-canceling} if for any set $S\sub V(G)$ of size less than $k$, we have $W_\sig(G-S)=0$.  We say that a graph $G$ is \textit{$k$-canceling} if there exists a $k$-canceling signing $\sig$ of $G$.  We will refer to 1-canceling graphs (i.e. those with $W_\sig(G)=0$ for some $\sig$) simply as \textit{canceling} graphs.
\end{defn}    

In Section~\ref{sec:k} we construct several families of graphs which are $k$-canceling for any given value of $k$.  We postpone describing these results, and instead focus on the nicest cases of $k=1,2$.  To state our main result, we recall that the square of a graph $G$, denoted $G^2$, is the graph which has $V(G^2)=V(G)$ and $uv\in G^2$  if and only if $d_G(u,v)\le 2$.  The following result implies, in particular, that squares of connected graphs on at least 5 vertices are canceling.
\begin{thm}\label{thm:square}
	Let $G$ be an $n$-vertex graph with $n\ge 5$.  If $G$ contains the square of an $n$-vertex tree, then $G$ is canceling.  If $G$ contains the square of a Hamiltonian cycle, then $G$ is 2-canceling.
\end{thm}
It was proven by Fan and Kierstead~\cite{FK} that any $n$-vertex graph $G$ with minimum degree at least $2(n-1)/3$ contains the square of a Hamiltonian path, and hence this is a sufficient condition for $G$ to be canceling.  It has also been proven by Koml\'os, S\'ark\"ozy, and Szemer\'edi~\cite{KSS} that if $n$ is sufficiently large and $G$ has minimum degree at least $2n/3$, then $G$ contains the square of a Hamiltonian cycle.  This result, together with Theorem~\ref{thm:square} and the definition of being 2-canceling, gives the following.

\begin{cor}
	If $n$ is sufficiently large and $G$ is an $n$-vertex graph with minimum degree at least $2n/3$, then there exists a signing $\sig$ of $G$ such that $W_\sig(G)=W_\sig(G-v)=0$ for all $v\in V(G)$.
\end{cor}

Our second main result involves necessary conditions for a graph to be canceling.
\begin{thm}\label{thm:necessary}
	If $G$ is a canceling graph on $n\ge 2$ vertices, then the following conditions hold:
	\begin{enumerate}
		\item[(a)] $G$ is connected,
		\item[(b)] $G$ contains an odd cycle,
		\item[(c)] $G$ has minimum degree at least 2, and
		\item[(d)] $G$ has at least $n+2$ edges.
	\end{enumerate}
	Moreover, the conditions of (c) and (d) are best possible: for all $n\ge 5$, there exists an $n$-vertex graph $G$ with minimum degree 2 and $n+2$ edges which is canceling.
\end{thm}

Theorem~\ref{thm:necessary} and a simple induction argument gives the following result, where we recall that a graph $G$ is $k$-connected if $G-S$ is connected for all $S\sub V(G)$ of size less than $k$.

\begin{cor}
	If $G$ is a $k$-canceling graph on $n\ge k+1$ vertices, then the following  conditions must hold: 
	\item[(a)] $G$ is $k$-connected,
	\item[(b)] $G$ contains an odd cycle,
	\item[(c)] $G$ has minimum degree at least $k+1$, and
	\item[(d)] $G$ has at least $n+{k\choose 2}+2k$ edges.
\end{cor}

The rest of the paper is organized as follows.  Theorem~\ref{thm:square} is proven in Section~\ref{sec:square} and Theorem~\ref{thm:necessary} is proven in Section~\ref{sec:necessary}.  In Section~\ref{sec:k} we construct $k$-canceling graphs for all $k$, and in Section~\ref{sec:r} we briefly investigate an $r$-colored variant of $k$-canceling signed graphs.  We close with some concluding remarks in Section~\ref{sec:concluding}.

Before moving onto the main text, we record two basic observations that will be useful to us throughout the paper.
\begin{obs}\label{obs:span}
	If $G'$ is a spanning subgraph of $G$ and $G'$ is $k$-canceling, then $G$ is $k$-canceling.
\end{obs}
This result follows by taking any signing $\sig$ of $G$ that restricts to a $k$-canceling signing of $G'$.  A slightly less obvious observation is the following.
\begin{obs}\label{obs:k}
	If $G$ has at least $k+1$ vertices, then a signing $\sig$ is $k$-canceling if and only if $W_\sig(G-S)=0$ for all sets $S\sub V(G)$ of size exactly $k-1$.
\end{obs}
\begin{proof}
	The forward direction is clear, so assume $W_\sig(G-S)=0$ for all sets $S\sub V(G)$ of size $k-1$, and let $T\sub V(G)$ be a set of size $\ell<k$.  We claim that $d_{G-T,\sig}(u,v)=0$ for any $u,v\in V(G-T)$.  Indeed, take any $S\supset T$ of size $k-1$ which does not contain $u,v$ (which exists since $G$ has at least $k+1$ vertices).  Then by assumption of $W_\sig(G-S)=0,$ there exists a $uv$-path $P$ in $G-S$ from $u$ to $v$ with $\sig(P)=0$.  This path is also in $G-T$, which shows that $d_{G-T,\sig}(u,v)=0$.  As $T,u,v$ were arbitrary, we conclude that $\sig$ is $k$-canceling.
\end{proof}

\section{Proof of Theorem~\ref{thm:square}}\label{sec:square}
In order to prove Theorem~\ref{thm:square}, we first show that a particular signing for squares of path graphs gives $W_\sig (P_n^2)=0$ for $n\ge 5$ and $n\ne 6$.
\begin{lem}\label{lem:squarePath}
	Let $P_n$ be the $n$-vertex path on $v_1,\ldots,v_n$ and let $\sig$ be the signing of $P_n^2$ such that $\sig(e)=+1$ if and only if $e\in P_n$.   If $n\ge 5$, then we have $d_\sig(v_i,v_j)=0$ for all $i,j$ unless $n=6$ and $\{i,j\}=\{1,6\}$
\end{lem}
\begin{proof}
	Assume $n\ge 5$ and let $v_i,v_j\in V(P_n^2)$ with $i<j$.  We first claim that if $j-i\equiv 0\mod 3$ (possibly with $j=i$), then there exists a path $P_{i,j}$ from $v_i$ to $v_j$ such that $\sig(P_{i,j})=0$ and such that every $v_k\in V(P_{i,j})$ satisfies $i\le k\le j$ and $k-i\equiv 0,1\mod 3$.  Indeed, this follows by taking the path \[v_iv_{i+1}v_{i+3}v_{i+4}v_{i+6}\cdots v_{j-3}v_{j-2}v_j.\]
	
	The path $P_{i,j}$ shows that $d_\sig(v_i,v_j)=0$ if $j-i\equiv 0\mod 3$.  If $j-i\equiv 1\mod 3$  and $j\ne 2$, then one can take the path formed by $P_{i,j-1}$ concatenated with the path $v_{j-1}v_{j-2}v_j$, and we note that $1\le j-2\le n$ since $j\ne 2$ and that $v_{j-2}\notin P_{i,j-1}$ since $(j-2)-i\equiv 2\mod 3$.  If $j=2$ then we must have $i=1$ and we can take the path $v_1v_3v_2$ since $n\ge 3$.
	
	If $j-i\equiv 2\mod 3$ and $j\ne 3,n$, then we can take $P_{i,j-2}$ together with $v_{j-2}v_{j-3}v_{j-1}v_{j+1}v_j$.  If $j=3$ then $i=1$ and we can take the path $v_1v_2v_4v_5v_3$ because $n\ge 5$.  Thus it remains to show that $d_\sig(v_i,v_n)=0$ when $n-i\equiv 2\mod 3$.  By the symmetry of the signing, it suffices to show $d_\sig(v_1,v_j)=0$ when $j-1\equiv 2\mod 3$, and by the analysis we've already done this holds except possibly when $j=n$ and $n\equiv 0\mod 3$.  This proves the $n=6$ case, so we now only need to consider $n\ge 9$ with $n\equiv 0\mod 3$.  In this case we can show $d_\sig(v_1,v_n)=0$ by using the path $P=v_1v_3v_2v_4v_5v_6v_8v_7v_9$ (see Figure~\ref{fig:P1}) concatenated with $P_{9,n}$ (and by our claim, $P_{9,n}$ uses no internal vertices of $P$), proving the lemma.
	\begin{figure}
	\[
	\begin{tikzpicture}[scale=1]
		\node at (1,0) { $\bullet$};
		\node at (2,0) { $\bullet$};
		\node at (3,0) { $\bullet$};
		\node at (4,0) { $\bullet$};
		\node at (5,0) { $\bullet$};
		\node at (6,0) { $\bullet$};
		\node at (7,0) { $\bullet$};
		\node at (8,0) { $\bullet$};
		\node at (9,0) { $\bullet$};
		\node at (1,-.5) {$v_1$};
		\node at (2,-.5) {$v_2$};
		\node at (3,-.5) {$v_3$};
		\node at (4,-.5) {$v_4$};
		\node at (5,-.5) {$v_5$};
		\node at (6,-.5) {$v_6$};
		\node at (7,-.5) {$v_7$};
		\node at (8,-.5) {$v_8$};
		\node at (9,-.5) {$v_9$};
		
		\draw[dashed] (3,0) arc [radius=1, start angle=0, end angle=180];
		\draw (3,0)--(2,0);
		\draw[dashed] (4,0) arc [radius=1, start angle=0, end angle=180];
		\draw (4,0)--(5,0);
		\draw (5,0)--(6,0);
		\draw[dashed] (8,0) arc [radius=1, start angle=0, end angle=180];
		\draw[dashed] (9,0) arc [radius=1, start angle=0, end angle=180];
		\draw (8,0)--(7,0);
	\end{tikzpicture}
	\]
	\caption{The $v_1v_9$-path $P$ with $\sig(P)=0$.}\label{fig:P1}
\end{figure}
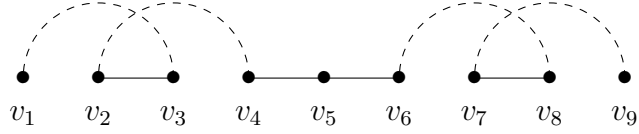
\end{proof}

Let $S_n$ denote the $n$-vertex star.  In this case $S_n^2=K_n$, so to prove Theorem~\ref{thm:square}, we will need the following lemma which shows that (in particular) $K_n$ is canceling for $n\ge 5$.

\begin{lem}\label{lem:Kn}
	If $n\ge 5$, then $K_n$ is 2-canceling.
\end{lem}
\begin{proof}
	Let $V(K_n)=\{v_1,\ldots,v_n\}$ and consider the signing $\sig$ which has $\sig(v_iv_{i+1})=+1$ for all $1\le i\le n$ (with the indices written mod $n$) and $\sig(e)=-1$ for all other edges $e$.  By Observation~\ref{obs:k} and the symmetry of the problem, it suffices to show that $W_\sig(K_n-S)=0$ where $S=\{v_n\}$.
	
	Consider any $v_i,v_j\in K_n-S$ with $i<j$.  If $i\ne 1$ then the path $v_iv_{i-1}v_j$ in $K_n-S$ shows that $d_\sig(v_i,v_j)=0$, and if $j\ne n-1$ then the path $v_iv_{j+1}v_j$ works.  Thus it suffices to show $d_\sig(v_1,v_{n-1})=0$, and this works by considering $v_1v_2v_{n-1}$ because $n\ge 5$.
\end{proof}
We now move on to prove our general theorem.

\begin{proof}[Proof of Theorem~\ref{thm:square}]
	By Observation~\ref{obs:span} it suffices to prove the result when $G=C_n^2$ or $G$ is the square of an $n$-vertex tree.  One can check by hand that the signings in Figure~\ref{fig:squareSign} show that $G=C_7^2$ is 2-canceling and $G=P_6^2$ is canceling.
	\begin{figure}
		\[
		\begin{tikzpicture}
			\node at (0,0) {$\bullet$};
			\node at (-2,-1) {$\bullet$};
			\node at (-2,-2) {$\bullet$};
			\node at (2,-1) {$\bullet$};
			\node at (2,-2) {$\bullet$};
			\node at (-1,-3) {$\bullet$};
			\node at (1,-3) {$\bullet$};
			
			\draw[dashed] (0,0)--(2,-1);
			\draw[dashed] (0,0)--(2,-2);
			\draw[dashed] (1,-3)--(2,-1);
			\draw[dashed] (-1,-3)--(-2,-2);
			\draw[dashed] (-2,-1)--(-2,-2);
			\draw (2,-1)--(2,-2);
			\draw (2,-2)--(1,-3);
			\draw (1,-3)--(-1,-3);
			\draw (0,0)--(-2,-1);
			\draw (0,0)--(-2,-2);
			\draw (1,-3)--(-2,-2);
			\draw (-2,-1)--(-1,-3);
			\draw (-1,-3)--(2,-2);
			\draw (2,-1)--(-2,-1);

			\node at (4.8,0) {$\bullet$};
			\node at (6.2,0) {$\bullet$};
			\node at (7,-1.5) {$\bullet$};
			\node at (6.2,-3) {$\bullet$};
			\node at (4.8,-3) {$\bullet$};
			\node at (4,-1.5) {$\bullet$};
			
			\draw[dashed] (6.2,0)--(6.2,-3);
			\draw[dashed] (6.2,-3)--(4,-1.5);
			\draw[dashed] (7,-1.5)--(4.8,-3);
			\draw (6.2,0)--(7,-1.5);
			\draw (7,-1.5)--(6.2,-3);
			\draw (4.8,-3)--(6.2,-3);
			\draw (4.8,-3)--(4,-1.5);
			\draw (4,-1.5)--(4.8,0);
			\draw (4.8,0)--(4.8,-3);
		\end{tikzpicture}
		\]
		\caption{Signings showing that $C_7^2$ is 2-canceling and $P_6^2$ is canceling.}\label{fig:squareSign}
	\end{figure}
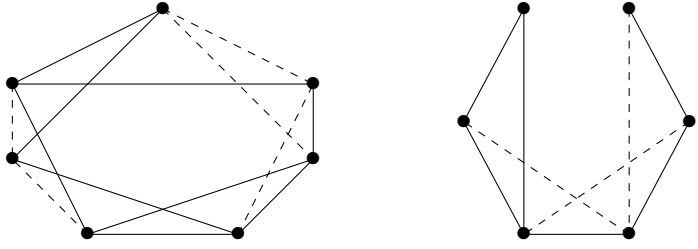
	
	Consider the case $G=C_n^2$.  Note that $C_5^2=K_5$, so this is 2-canceling by Lemma~\ref{lem:Kn}.  Thus it remains to prove the result for $n\ge 6$ with $n\ne 7$. In this case we let $\sig$ be the signing with $\sig(e)=+1$ if and only if $e\in C_n$.  By Observation~\ref{obs:k} it suffices to prove $W_\sig(C_n^2-v)$ for any vertex $v$.  Observe that $(C_n^2-v,\sig)$ contains as a spanning subgraph the signed graph $(P_{n-1}^2,\sig')$ with $\sig'$ the signing from Lemma~\ref{lem:squarePath}.  Because $W_{\sig'}(P_{n-1}^2)=0$ for $n\ge 6$ and $n\ne 7$, this implies that $W_\sig(C_n^2-v)=0$ for these values of $n$.
	
	It remains to prove the result when $G=T^2$ for some $n$-vertex tree $T$. If $T$ is the $n$-vertex star $S_n$, then $T^2=K_n$ and the result is given by Lemma~\ref{lem:Kn}.  Thus in what follows we assume that $T$ is an $n$-vertex tree with $n\ge 5$ such that $T$ is not $P_6$ nor a star $S_n$.
	
	Take $\sig$ to be the signing with $\sig(e)=+1$ if and only if $e\in T$.  Let $u,v$ be arbitrary vertices in $T$, and let $Q$ be any maximal path in $T$ containing $u,v$, say with $Q=x_1\cdots x_m$ and $x_i=u,x_j=v$.  Observe that the signed graph $(Q^2,\sig)$ is isomorphic to the signed graph $(P_m^2,\sig')$ in Lemma~\ref{lem:squarePath}, so by the lemma we have $d_\sig(u,v)=0$ unless possibly $m\le 4$ or $m=6$ and $\{i,j\}=\{1,6\}$.
	
	First note that $m\ge 3$ since $Q$ is a maximal path and $T$ is a tree on at least 3 vertices.  If $Q=x_1x_2x_3$, then because $n\ge 4$, there must exist some vertex $w\notin Q$ which is adjacent to some $x_i$ vertex.  Note that $x_1,x_3$ are leaves because $Q$ is maximal, so we must have that $w$ is adjacent to $x_2$.  Because $n\ge 5$ and the assumption that no vertex in $T$ has degree $n-1$ (i.e., $T\ne S_n$), some vertex $w'$ must be at distance 2 from $x_2$ in $T$, and without loss of generality we can assume $w'$ is adjacent to $w$, see Figure~\ref{fig:Q}(a).  We can then consider the paths $x_1wx_2,\ x_2wx_3$, and $x_1ww'x_2x_3$ to see that $d_\sig(x_a,x_b)=0$ for all $1\le a<b\le 3$, and in particular $d_\sig(u,v)=0$.
	
	\begin{figure}
		\[
		\begin{tikzpicture}
			\node at (-1,0) {$\bullet$};
			\node at (0,0) {$\bullet$};
			\node at (1,0) {$\bullet$};
			\node at (0,1) {$\bullet$};
			\node at (0,2) {$\bullet$};
			\node at (-1.2,-.5) {$x_1$};
			\node at (0,-.5) {$x_2$};
			\node at (1.2,-.5) {$x_3$};
			\node at (-.5,1) {$w$};
			\node at (-.5,2) {$w'$};
			\node at (0,-2) {$(a)$};
			
			\draw (-1,0)--(0,0);
			\draw (1,0)--(0,0);
			\draw (0,1)--(0,0);
			\draw (0,2)--(0,0);
			\draw[dashed] (-1,0)--(0,1);
			\draw[dashed] (1,0)--(0,1);
			\draw[dashed] (0,2) arc [radius=1, start angle=90, end angle=-90];
			\draw[dashed] (-1,0) arc [radius=1, start angle=180, end angle=360];

			\node at (4,0) {$\bullet$};
			\node at (5,0) {$\bullet$};
			\node at (6,0) {$\bullet$};
			\node at (7,0) {$\bullet$};
			\node at (3.8,-.5) {$x_1$};
			\node at (5,-.5) {$x_2$};
			\node at (6.2,-.5) {$x_3$};
			\node at (7,-.5) {$x_4$};
			\node at (5,1) {$\bullet$};
			\node at (4.5,1) {$w$};
			\node at (5.5,-2) {$(b)$};
			
			\draw (4,0)--(5,0);
			\draw (6,0)--(5,0);
			\draw (6,0)--(7,0);
			\draw (5,1)--(5,0);
			\draw[dashed] (4,0)--(5,1);
			\draw[dashed] (6,0)--(5,1);
			\draw[dashed] (4,0) arc [radius=1, start angle=180, end angle=360];
			\draw[dashed] (5,0) arc [radius=1, start angle=180, end angle=0];
		\end{tikzpicture}
		\]
		\caption{(a) The case $Q=x_1x_2x_3$.  (b) The case $Q=x_1x_2x_3x_4$.}\label{fig:Q}
	\end{figure}
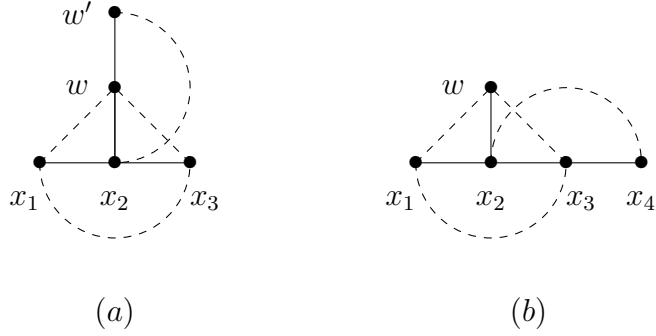            
	
	Now assume $Q=x_1x_2x_3x_4$.  Again $n\ge 5$ implies that there must exist some other vertex $w$, and by the symmetry of the problem (and that $x_1,x_4$ are leaves) we can assume $w$ is adjacent to $x_2$, see Figure~\ref{fig:Q}(b).   Now we can use the paths \[x_1wx_2,\ x_1wx_2x_4x_3,\ x_1x_3x_4,\ x_2x_1x_3,\ x_2x_1wx_3x_4,\ x_3x_2x_4\]
	to see that again $d_\sig(u,v)=0$.
	
	Finally, we consider the case $Q=x_1\cdots x_6$ with, say, $u=x_1,v=x_6$.  Because $T\ne P_6$ there exists some other vertex $w$, which without loss of generality we can assume is adjacent to either $x_2$ or $x_3$.  In the former case we consider the path $x_1wx_2x_4x_3x_5x_6$, and in the latter case we use $x_1x_2wx_3x_5x_4x_6$.  Either case shows $d_\sig(u,v)=0$, completing the proof.
\end{proof}

Before closing this section, we note that the condition on $n$ in Theorem~\ref{thm:square} is tight:  $P_4^2$ is $K_4$ minus an edge which is not canceling, and similarly $C_4^2=K_4$ is not 2-canceling.

\section{Proof of Theorem~\ref{thm:necessary}}\label{sec:necessary}
It is clear that a canceling graph must be connected.  This and the following lemma shows that bipartite graphs on at least 2 vertices can not be canceling.
\begin{lem}\label{lem:bipartite}
	If $G$ is a bipartite graph with bipartition $U\cup V$, then $W_\sig(G)\ge |U||V|$ for every signing $\sig$.
\end{lem}
\begin{proof}
	Let $\sig$ be a signing of $G$, $u\in U$, and $v\in V$.  We claim that $d_\sig(u,v)\ge 1$.  Indeed, because $G$ is bipartite, every path from $u$ to $v$ in $P$ has an odd number of edges, and hence $|\sig(P)|$ will be an odd integer.  In particular, it will be at least 1.  Thus
	\[W_\sig(G)=\half \sum_{x,y\in V(G)} d_\sig(x,y)\ge \sum_{u\in U,\ v\in V} d_\sig(u,v)\ge |U||V|,\]
	proving the lemma.
\end{proof}
A similar argument gives the following.
\begin{lem}\label{lem:deg1}
	If $G$ has $\ell$ vertices of degree 1, then $W_\sig(G)\ge \ell$ for every signing $\sig$.
\end{lem}
\begin{proof}
	If $u$ has degree 1, let $x_u$ be its unique neighbor.  Observe that $ux_u$ is the only path from $u$ to $x_u$, so $d_\sig(u,x_u)=1$ for every signing $\sig$.  Summing this over all $u$ of degree 1 gives the result.
\end{proof}

Lemma~\ref{lem:deg1} shows that we only have to consider (connected) graphs with minimum degree at least 2.  The following lemma with $t=2$ shows that canceling graphs must contain a vertex of degree at least 3.

\begin{lem}\label{lem:theta}
	Let $G$ be a graph which consists of internally disjoint paths $P_1,\ldots,P_t$ that have the same endpoints $x$ and $y$.  If $t\le 3$, then $G$ is not canceling.  
\end{lem}
This bound on $t$ is best possible: Figure~\ref{fig:theta} gives a signing of a graph consisting of 4 paths with common endpoints which is canceling.
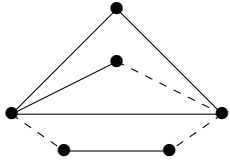
\begin{figure}
	\[
	\begin{tikzpicture}[scale=.7]
		\node at (0,0) {$\bullet$};
		\node at (4,0) {$\bullet$};
		\node at (2,1) {$\bullet$};
		\node at (1,-.7) {$\bullet$};
		\node at (3,-.7) {$\bullet$};
		\node at (2,2) {$\bullet$};
		
		\draw (0,0)--(4,0);
		\draw (0,0)--(2,1);
		\draw[dashed] (2,1)--(4,0);
		\draw[dashed] (0,0)--(1,-.7);
		\draw[dashed] (4,0)--(3,-.7);
		\draw (1,-.7)--(3,-.7);
		\draw (2,2)--(0,0);
		\draw (2,2)--(4,0);
	\end{tikzpicture}
	\]
	\caption{A canceling graph which consists of four paths with the same endpoints.}\label{fig:theta}
\end{figure}

\begin{proof}
	Assume for contradiction that there exists a signing $\sig$ of $G$ such that $d_\sig(u,v)=0$ for all $u,v\in V(G)$.  Intuitively our goal is to prove that $G$ must look roughly like the graph in Figure~\ref{fig:theta} after deleting the topmost path, which is not canceling because the two vertices in the bottom path have positive signed distance.
	
	It is not too hard to see that the only paths between $x$ and $y$ are the $P_i$ paths, so having $d_\sig(x,y)=0$ implies that $\sig(P_i)=0$ for some $i$, say with $i=1$.  Note that this means $P_1$ has a (positive) even number of edges.  
	
	Let $x'$ be the vertex in $P_1$ adjacent to $x$, and without loss of generality we assume $\sig(xx')=+1$.  It is not hard to see that the only paths from $x'$ to $x$ are the edge $xx'$ and the path $P_1-x$ concatenated with another one of the $P_j$ paths.  In particular, we must have $\sig(P_1-x)+\sig(P_j)=0$ for some $j\ne 1$, say $j=2$. Because $\sig(P_1)=0$ and $\sig(xx')=+1$, we have $\sig(P_1-x)=-1$ and hence $\sig(P_2)=+1$.  Similarly the only paths from $x'$ to $y$ are $P_1-x$ (which has $\sig(P_1-x)=-1\ne 0$) and paths which use $x'x$ together with one of the $P_{j'}$ paths, and the same reasoning as before implies that $\sig(P_3)=-1$.
	
	Because $G$ is a graph (and not a multigraph), at least one of $P_2$ and $P_3$ uses more than 1 edge.  Without loss of generality we can assume this holds for $P_3$.  Let us denote the vertices of $P_3$ by $z_1z_2\cdots z_p$, where $z_1=x$ and $z_p=y$.  Because $P_3$ has at least two edges and $\sig(P_2)=-1$, there must exist some $j$ such that $\sig(z_jz_{j+1})=+1$.  Observe that every path from $z_j$ to $z_{j+1}$ is either $z_jz_{j+1}$, or $z_jz_{j-1}\cdots z_1$ concatenated with some $P_{j'}$ and the path $z_{p}z_{p-1}\cdots z_{j+1}$.  Note that $\sig(z_jz_{j+1})=+1\ne 0$, and the $z_jz_{j+1}$-path $P$ using $P_{j'}$ has \[\sig(P)=\sig(P_{j'})+\sig(P_3)-\sig(z_jz_{j+1})=\sig(P_{j'})-2.\]  Because $\sig(P_1),\sig(P_2),\sig(P_3)\ne 2$, we have $d_\sig(z_j,z_{j+1})>0$, a contradiction.
\end{proof}

The last lemma we will need to show that the conditions of Theorem~\ref{thm:necessary} are necessary is the following.

\begin{lem}\label{lem:sharedVertex}
	Let $G_1,G_2$ be graphs such that $V(G_1)\cap V(G_2)=\{w\}$.  Then $G_1\cup G_2$ is canceling if and only if $G_1$ and $G_2$ are canceling.
\end{lem}
\begin{proof}
	For any signing $\sig$, we claim that
	\[W_\sig(G_1\cup G_2)\ge W_\sig(G_1)+W_\sig(G_2).\]
	Indeed, consider any two vertices $u,v\in V(G_1)$.  We observe that every path between $u$ and $v$ must lie entirely in $G_1$, so the signed distance between these two vertices in $G_1\cup G_2$ is the same as their signed distance restricted just to $G_1$.  The same result holds for $G_2$, so in total this gives
	\[W_\sig(G_1\cup G_2)\ge \half \sum_{i=1,2}\sum_{u,v\in V(G_i)} d_\sig(u,v)=W_\sig(G_1)+W_\sig(G_2).\]
	In particular, if there exists no signing with, say, $W_\sig(G_1)=0$, then the inequality above shows that there exists no such signing for $G_1\cup G_2$ as well.
	
	Now assume there exist signings $\sig_i$ for $G_i$ such that $W_{\sig_i}(G_i)=0$.  Let $\sig$ be the signing of $G_1\cup G_2$ which has $\sig(e)=\sig_i(e)$ for $e\in G_i$ (this is well defined since $G_1,G_2$ have no edges in common).  With this we immediately have that $d_\sig(u,v)=0$ for $u,v\in V(G_i)$ and any $i$.  If $u_i\in V(G_i)$, then by assumption there exists a $u_iw$-path $P_i$ in $G_i$ with $\sig_i(P_i)=0$.  The concatenation of $P_1$ and $P_2$ shows that $d_\sig(u_1,u_2)=0$ for any $u_i\in V(G_i)$, and thus $W_\sig(G_1\cup G_2)=0$ as desired.
\end{proof}

We use the following lemma to construct a family of graphs which show that the conditions of Theorem~\ref{thm:necessary} are best possible.
\begin{lem}\label{lem:subdivision}
	Let $(G',\sig')$ be a signed graph with $W_{\sig'}(G')=0$ such that there exists a cycle $C'$ in $G'$ containing an edge $e$ with $\sum_{e'\in C'} \sig'(e')=-\sig'(e)$.  If $G$ is any graph obtained by replacing $e$ in $G'$ with a path on an odd number of edges, then $G$ is canceling.
\end{lem}
Examples of graphs and edges $e$ which satisfy these conditions can be found in Figure~\ref{fig:smalln}.
\begin{proof}
	Let $e=xy$ be an edge as in the hypothesis of the lemma and let $G$ be $G'$ after replacing $e$ with the path $P=xw_1\cdots w_{2i}y$ for some $i\ge 0$.  We let $w_0=x$ and $w_{2i+1}=y$.  Define a signing $\sig$ for $G$ by $\sig(e')=\sig'(e')$ if $e'\in G'$ and such that $\sig(w_jw_{j+1})=(-1)^{j}\sig'(e)$ for all $0\le j\le 2i$.  Observe that with this $\sgn(P)=\sgn'(e)$.   We claim that $W_\sig(G)=0$.
	
	Let $u,v\in V(G)$.  If $u,v\in V(G')$, then by assumption there exists some $uv$-path $Q'$ in $G'$ such that $\sig'(Q')=0$.  Let $Q$ be the path in $G$ obtained by replacing the edge $e$ in $Q'$ with the path $P$ (and if $e\notin Q'$ we let $Q=Q'$).  Because $\sgn(P)=\sgn'(e)$, we have $\sgn(Q)=\sgn'(Q')=0$, so $d_\sig(u,v)=0$ in this case.  
	
	Now consider the case $u=w_j$, and without loss of generality assume that $j$ is even.  If $v\in V(G')$, then there exists an $xv$-path $Q'_x$ in $G'$ with $\sig(Q'_x)=0$ (if $v=x$, then $Q'_x$ is the empty path).  If $Q'_x$ does not use the edge $e=xy$ then $P_{\le j}:=w_jw_{j-1}\cdots w_1 x$ concatenated with $Q'_x$ shows $d_\sig(u,v)=0$ (here $\sig(P_{\le j})=0$ because $j$ is even), and otherwise $e$ must be the first edge of $Q'_x$ (since the path starts at $x$) and one can concatenate $P_{\ge j}:=w_jw_{j+1}\cdots y$ with $Q'_x-e$ to show $d_\sig(u,v)=0$ by using $\sig(P_{\ge j})=\sig(P)-\sig(P_{\le j})=\sig'(e)$.

	Now assume $v=w_{j'}$. If $j'$ is even, then we can assume $j<j'$ and use the path $w_jw_{j+1}\cdots w_{j'}$ to show $d_\sig(w_j,w_{j'})=0$. Thus we can assume $v=w_{j'}$ with $j'$ odd.  If $j<j'$, then we can take the concatenation of the path $P_{\le j}$ together with $Q'_y$ (which does not use the edge $e$) and $P_{\ge j'}$ to show $d_\sig(w_j,w_{j'})=0$.  If $j'<j$, then take the concatenation of the paths $P_{\le j'}$, $Q:=C'-e$ (where $C'$ is as in the hypothesis), and $P_{\ge j}$.  Observe that \[\sig(P_{\le j'})+\sig(P_{\ge j})=\sig(P)-\sum_{t=j'}^{j-1} \sig(w_{t}w_{t+1})=\sig'(e)+\sig'(e)=2\sig'(e),\]
	where here we used that $j'$ is odd (so the first term of the sum is $-\sig'(e)$) and that $j$ is even (so that there are an odd number of terms in the alternating sum).  By hypothesis $\sig(Q)=-2\sig'(e)$, so in total this path gives $d_\sig(w_j,w_{j'})=0$, completing the proof.
\end{proof}

We now prove our main result for this section.

\begin{proof}[Proof of Theorem~\ref{thm:necessary}]
	Let $G$ be a canceling graph on at least 2 vertices.  It is clear that $G$ must be connected, and Lemmas~\ref{lem:bipartite} and \ref{lem:deg1} implies that $G$ contains an odd cycle and has minimum degree at least 2.  Given this, if $G$ had $n$ edges then $G$ would have to be a cycle, but this is impossible by Lemma~\ref{lem:theta} with $t=2$.  Thus $G$ must have at least $n+1$ edges.
	
	If $G$ had exactly $n+1$ edges, then because $G$ has minimum degree at least 2 it must either have exactly one vertex of degree 4 and the rest of degree 2, or exactly two vertices of degree 3 and the rest of degree 2.  It is not difficult to see that the former case (and $G$ being connected) implies that $G$ consists of two cycles sharing a common vertex, and this is not canceling by Lemmas~\ref{lem:sharedVertex} and \ref{lem:theta} for $t=2$.  The latter case implies that $G$ consists of three internally disjoint paths from the two vertices of degree 3, and in this case Lemma~\ref{lem:theta} with $t=3$ implies that $G$ is not canceling.  Thus $G$ must contain at least $n+2$ edges.

	To show that these conditions are best possible, consider the two graphs $G_0,G_1$ given in Figure~\ref{fig:smalln}.  It is not hard to see that these are canceling (in particular, $G_1=(P_5^2,\sig)$ with $\sig$ as in Lemma~\ref{lem:squarePath}), that the edges labeled $e$ satisfy the hypothesis of Lemma~\ref{lem:subdivision}, and that $|E(G_i)|=|V(G_i)|+2$ for $i=1,2$.  Thus for all even $n\ge 6$, one can use Lemma~\ref{lem:subdivision} to replace $e\in G_0$ by a path with $n-4$ new vertices to obtain  an $n$-vertex canceling graph $G$ which has minimum degree 2 and $n+2$ edges.  Similarly one can replace the edge $e\in G_1$ with a path to get the result for all odd $n\ge 5$, proving the result.
	\begin{figure}
		\[
		\begin{tikzpicture}
		\node at (0,0) {$\bullet$};
		\node at (1,0) {$\bullet$};
		\node at (0,1) {$\bullet$};
		\node at (1,1) {$\bullet$};
		\node at (-.4,.5) {$e$};
		\node at (.5,-1) {$G_1$};
		
		\draw (0,0)--(0,1);
		\draw (0,1)--(1,1);
		\draw (1,1)--(1,0);
		\draw[dashed] (0,0)--(1,1);
		\draw[dashed] (0,0)--(1,0);
		\draw[dashed] (1,0)--(0,1);

		\node at (3,0) {$\bullet $};
		\node at (3,1) {$\bullet $};
		\node at (5,0) {$\bullet $};
		\node at (5,1) {$\bullet $};
		\node at (4,2) {$\bullet$};
		\node at (3.5,.3) {$e$};
		\node at (4,-1) {$G_2$};
		
		\draw (3,0)--(3,1);
		\draw (3,1)--(4,2);
		\draw (4,2)--(5,1);
		\draw (5,1)--(5,0);
		\draw[dashed] (3,0)--(4,2);
		\draw[dashed] (3,1)--(5,1);
		\draw[dashed] (4,2)--(5,0);
		\end{tikzpicture}
		\]
		\caption{The graphs $G_1$ and $G_2$ and edges $e$ which satisfy Lemma~\ref{lem:subdivision}.}\label{fig:smalln}
	\end{figure}
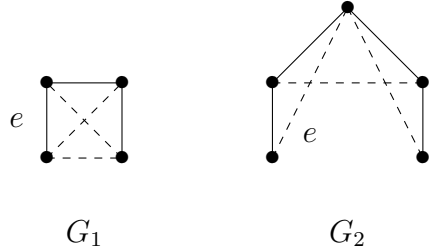
\end{proof}

\section{$k$-canceling Graphs}\label{sec:k}
By Observation~\ref{obs:span}, we see that there exists an $n$-vertex $k$-canceling graph if and only if $K_n$ is $k$-canceling, so as a first step it makes sense to try and understand when $K_n$ is $k$-canceling.  The best general bound we have for this is the following.  
\begin{prop}\label{prop:Knk}
	If $n\ge 2k+4$, then $K_n$ is $k$-canceling.
\end{prop}
We note that this bound is not tight, even for $k=1$.  Proposition~\ref{prop:Knk} is an immediate corollary of the following result by taking  $G'$ to be a complete bipartite graph on $U\cup V$ with $|U|=|V|=k+2$.
\begin{prop}
	Let $G'$ be a bipartite graph on $U\cup V$ with $|U|,|V|\ge k+2$ and minimum degree at least $k+1$.  Let $G$ be the graph obtained from $G'$ by adding every edge between two vertices of $U$ and every edge between two vertices of $V$.  Then $G$ is $k$-canceling.
\end{prop}
\begin{proof}
	Let $\sig$ be the signing such that $\sig(uv)=+1$ if and only if $u\in U$ and $v\in V$.  Let $x,y$ be arbitrary vertices of $G$ and $S$ a set of less than $k$ vertices.
	
	First consider the case $x\in U$ and $y\in V$.  Because $x$ has at least $k+1$ neighbors in $V$, it has some neighbor $z\in V-S-\{y\}$, and in this case the path $xzy$ shows $d_\sig(x,y)=0$.  Now suppose $x,y\in U$.  Because $|U-S|\ge 3$, there exists some $z\in U-S-\{x,y\}$.   As argued before, $z$ has a neighbor $w\in V-S$ and $y$ has a neighbor $w'\in V-S-\{w'\}$.  Then the path $xzww'y$ shows $d_\sig(x,y)=0$.  A symmetric argument gives the result if $x,y\in V$, completing the proof.
\end{proof}

Another family of examples can be obtained from blowups of odd cycles.  If $G$ is a graph on $\{v_1,\ldots,v_t\}$, then the \textit{$(n_1,\ldots,n_t)$-blowup} of $G$ is defined to be the $t$-partite graph on sets $V_1,\ldots,V_t$ with $|V_i|=n_i$ and with $u\in V_i$ and $w\in V_j$ adjacent if and only if $v_i\sim v_j$ in $G$.

\begin{prop}
	Let $G$ be the $(n_1,\ldots,n_{2t+1})$-blowup of a cycle $C_{2t+1}$ with $t\ge 1$.  If $n_i\ge 2k$ for all $i$, then $G$ is $k$-canceling.
\end{prop}
By Theorem~\ref{thm:necessary}(b), it is necessary that we take the blowup of an odd cycle. The uniform bound $n_i\ge 2k$ cannot be weakened when $k=1$ since Theorem~\ref{thm:necessary}(d) shows that $C_{2t+1}$ (which is the $(1,\ldots,1)$-blowup of $C_{2t+1}$) is not canceling.  It may be possible to improve the bound on $n_i$ for larger $k$.
\begin{proof}
	Let $V_1\cup \cdots \cup V_{2t+1}$ be the partition of $G$ such that $|V_i|=n_i$ and every vertex of $V_i$ is adjacent to every vertex of $V_{i-1}\cup V_{i+1}$.  Further partition each set $V_i$ arbitrarily into two sets $V_i^0,V_i^1$ each of size at least $k$.  Let $\sig$ be the signing which has $\sig(uv)=+1$ if and only if $u\in V_i^j$ and $v\in V_{i+1}^j$ for some $i,j$.
	
	Let $u,v\in V(G)$ and let $S$ be an arbitrary set of less than $k$ vertices.   Define $U_i^j=V_i^j-S$, and note that each of these sets contain some vertex $u_i^j$.  By relabeling our parts we can assume $u\in U_1^0$.   If $v\in U_1^1$, then the path $uu_{2}^0u_1^{1}u_{2}^1v$ shows $d_\sig(u,v)=0$, see Figure~\ref{fig:U}(a). If $v\in U_1^0$, then the path $uu_{2}^0v$ works.
	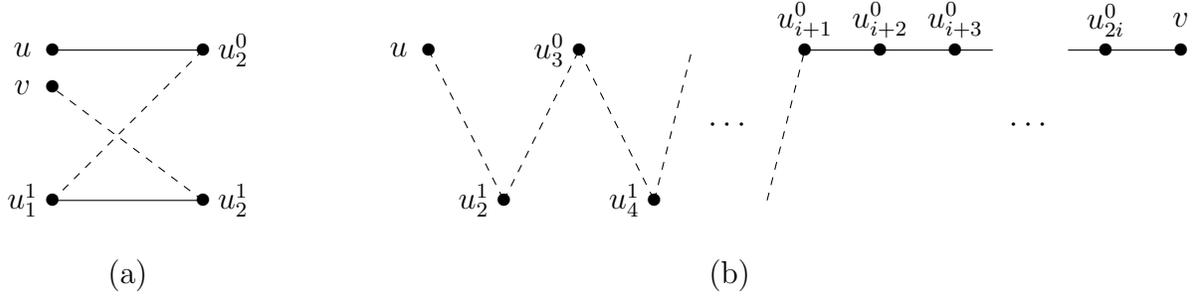
\begin{figure}
		\[\begin{tikzpicture}
			\node at (0,0) {$\bullet$};
			\node at (0,-.5) {$\bullet $};
			\node at (2,0) {$\bullet$};
			\node at (2,-2) {$\bullet$};
			\node at (0,-2) {$\bullet$};
			\node at (-.4,0) {$u$};
			\node at (-.4,-.5) {$v$};
			\node at (-.4,-2) {$u_1^1$};
			\node at (2.4,0) {$u_2^0$};
			\node at (2.4,-2) {$u_2^1$};
			\node at (1,-3) {(a)};
			
			\draw (0,0)--(2,0);
			\draw[dashed] (2,0)--(0,-2);
			\draw (0,-2)--(2,-2);
			\draw[dashed] (2,-2)--(0,-.5);
			
			\node at (5,0) {$\bullet$};
			\node at (4.6,0) {$u$};
			\node at (6,-2) {$\bullet$};
			\node at (5.6,-2) {$u_2^1$};
			\node at (7,0) {$\bullet$};
			\node at (6.6,0) {$u_3^0$};
			\node at (8,-2) {$\bullet$};
			\node at (7.6,-2) {$u_4^1$};
			\node at (9,-1) {$\cdots$};
			\node at (10,0) {$\bullet$};
			\node at (10,.4) {$u_{i+1}^0$};
			\node at (11,0) {$\bullet$};
			\node at (11,.4) {$u_{i+2}^0$};
			\node at (12,0) {$\bullet$};
			\node at (12,.4) {$u_{i+3}^0$};
			\node at (13,-1) {$\cdots$};
			\node at (14,0) {$\bullet$};
			\node at (14,.4) {$u_{2i}^0$};
			\node at (15,0) {$\bullet$};
			\node at (15,.4) {$v$};
			\node at (9,-3) {(b)};
			
			\draw[dashed] (5,0)--(6,-2);
			\draw[dashed] (6,-2)--(7,0);
			\draw[dashed] (7,0)--(8,-2);
			\draw[dashed] (8,-2)--(8.5,0);
			\draw[dashed] (9.5,-2)--(10,0);
			\draw (10,0)--(11,0);
			\draw (12,0)--(11,0);
			\draw (12,0)--(12.5,0);
			\draw (13.5,0)--(14,0);
			\draw (14,0)--(15,0);
		\end{tikzpicture}\]
	\caption{(a) The path $uu_{2}^0u_1^{1}u_{2}^1v$ when $u,v\in U_1^0$.  (b) The path to $v\in U_{2i+1}^0$ with $i$ even.}\label{fig:U}
	\end{figure}
,	
	
	Thus we can assume $v\in U_i^j$ for some $i\ne 1$, and possibly be relabeling our parts again we can assume\footnote{This is essentially because there is a path of even length between any two vertices of $C_{2t+1}$.}  $v\in U_{2i+1}^j$ for some $i$.  If $j$ has the same parity as $i$, then any path of the form\footnote{Here if $p\le i+1$, then the $p$th vertex of the path appears in $U_{p}^{q}$ with $p,q$ having opposite parities, so in particular the $i+1$st vertex will be in $U_{i+1}^j$ as we have implicitly claimed.} \[uu_2^1u_3^0u_4^1u_5^0\cdots u_i^{1-j}u_{i+1}^ju_{i+2}^j\cdots u_{2i}^jv\] 
	shows $d_\sig(u,v)=0$ since the first $i$ edges are all given negative signs while the latter are given positive signs, see Figure~\ref{fig:U}(b).  If $i,j$ have different parities, then one can use a path $P$ as above to reach $u_{2i+1}^{1-j}$, and then one can go to $u_{2i}^j$ and then to $v$ (note that the only vertex of $P$ in $U_{2i}$ is in $U_{2i}^{1-j}$, so $u_{2i}^j$ is a new vertex and this genuinely defines a path).  This completes the proof.
\end{proof}

\section{$r$-colored Graphs}\label{sec:r}
The study of canceling graphs has a natural extension to $r$-colored graphs.
\begin{defn}
	Given an edge coloring $\chi:E(G)\to [r]$, we say that a path $P$ is \textit{canceling} if $|\{e\in P:\chi(e)=i\}|=|\{e\in P:\chi(e)=j\}|$ for all $i,j$.  That is, every color is used the same number of times in $P$.  We say that an edge coloring $\chi:E(G)\to [r]$ is \textit{$(r,k)$-canceling} if for any $u,v\in V(G)$ and set $S\sub V(G)\sm\{u,v\}$ of size less than $k$, there exists a canceling path $P$ from $u$ to $v$ in $G-S$.  We say that a graph $G$ is \textit{$(r,k)$-canceling} if there exists an $(r,k)$-canceling edge coloring $\chi$ of $G$.
\end{defn}
Observe that $(2,k)$-canceling graphs are exactly $k$-canceling graphs since $d_\sig(u,v)=0$ is equivalent to the existence of a $uv$-path which uses each ``color'' $\pm 1$ the same number of times.  

A few of the basic lemmas we have established for $k$-canceling graphs extend to $(r,k)$-canceling graphs, but for the most part the proofs of our main results do not carry over to this setting.  For example, it may be the case that $P_n^r$ is $(r,1)$-canceling for $n$ sufficiently large in terms of $r$, but the natural analog of the coloring of Lemma~\ref{lem:squarePath} (namely, the one which gives two vertices color $i$ if they are at distance $i$ in $P_n$) can quickly be shown not to work for $r=3$, so new ideas would be needed here.

As a first step towards exploring these problems, we establish a bound for $K_n$.

\begin{prop}\label{prop:Knr}
	If $n\ge 3(k-1)(r-1)$ with $r\ge 3$ and $k\ge 2$, then $K_n$ is $(r,k)$-canceling.
\end{prop}
\begin{proof}
	Let $m=3(k-1)(r-1)$ and let $V(K_n)=\{x_1,\ldots,x_{m},y_{m+1},\ldots,y_n\}$.  For $i<m$ let $\chi(x_ix_{i+1})=j$ where $j\equiv i\mod(r-1)$, let $\chi(x_mx_1)=r-1$, and let $\chi(e)=r$ for all other edges.   That is, we form a cycle on the $x_{1},\ldots,x_{m}$ vertices which cycles through the first $r-1$ colors a total of $3(k-1)$ times, and we use the last color $r$ for every other edge, see Figure~\ref{fig:r} for the case $r=3$, $k=2$, and $n=6$.
	\begin{figure}
		\[\begin{tikzpicture}
			\node at (0,0) {$\bullet$};
			\node at (1,0) {$\bullet$};
			\node at (-.5,1) {$\bullet$};
			\node at (1.5,1) {$\bullet$};
			\node at (0,2) {$\bullet$};
			\node at (1,2) {$\bullet$};
			
			\draw (0,0)--(1,0);
			\draw[dashed] (1,0)--(1.5,1);
			\draw (1.5,1)--(1,2);
			\draw[dashed] (1,2)--(0,2);
			\draw (0,2)--(-.5,1);
			\draw[dashed] (-.5,1)--(0,0);
			\draw[densely dotted] (0,0)--(0,2);
			\draw[densely dotted] (0,0)--(1.5,1);
			\draw[densely dotted] (0,0)--(1,2);
			\draw[densely dotted] (1,0)--(0,2);
			\draw[densely dotted] (1.5,1)--(0,2);
			\draw[densely dotted] (1.5,1)--(-.5,1);
			\draw[densely dotted] (1,0)--(-.5,1);
			\draw[densely dotted] (1,2)--(-.5,1);
			\draw[densely dotted] (1,0)--(1,2);
		\end{tikzpicture}\]
	\caption{A 3-colored $K_6$.}\label{fig:r}
	\end{figure}
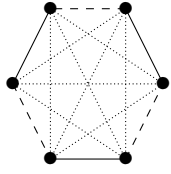  
	
	Let $u,v\in V(G)$ and let $S\not\ni u,v$ be a set of less than $k$ vertices.   For ease of presentation we will only consider the case that $S$ contains a vertex of $\{x_1,\ldots,x_{m}\}$ (the other case being strictly easier to analyze), and without loss of generality we can assume $x_{m}\in S$. Write $S\cap \{x_1,\ldots,x_m\}=\{x_{i_1},x_{i_2},\ldots,x_{i_p}\}$ with $i_j<i_{j+1}$ for all $j$ (and hence $i_p=m$).  By letting $i_0:=0$, we see that $\sum_{j=1}^p i_{j}-i_{j-1}=m$, so one of these at most $k-1$ terms must have size at least $m/(k-1)\ge 3r-3$.  Without loss of generality we can assume this holds for $j=1$, which means that $x_1,\ldots,x_{3r-4}\notin S$.  In particular, $x_1,\ldots,x_{2r-1}\notin S$ since $r\ge 3$, so we can always use these vertices to construct canceling paths in $K_n-S$.
	
	First assume\footnote{Note that we could have $x_{2r}\in S$, but in this case we trivially have $u,v\ne x_{2r}$.} $u,v\notin\{x_1,x_2,\ldots,x_{2r}\}$.  In this case the path $ux_1\cdots x_{2r-1}v$ is $r$-canceling (in particular, $u\ne x_{m}\in S$, so $\chi(ux_1)=r$).   If $u=x_i$ with $1\le i\le r$ and $v\notin \{x_1,\ldots,x_{2r}\}$, then one can consider the path $x_ix_{i+1}\cdots x_{i+r-1}v$.  If $u=x_i$ with $r\le i\le 2r$ and $v\notin \{x_1,\ldots,x_{2r}\}$, then one can consider the path $x_{i}x_{i-1}\cdots x_{i-r+1}v$.

	It remains to deal with the case $u=x_i,v=x_j$ with $i,j\in \{1,2,\ldots,2r\}$.  Without loss of generality we can assume $i<j$.   If $i\ge r$ then we can consider the path $x_ix_{i-1}\cdots x_{i-r+1}x_j$, and similarly if $j\le 2r-3$ we can consider the path $x_jx_{j+1}\cdots x_{j+r-1}x_i$.  Thus we can assume $i<r\le 2r-3<j$.  If $i+r<j$ then we could consider the path $x_i\cdots x_{i+r-1}x_j$, so in total we must have \[2r-2\le j\le i+r\le 2r-1.\]
	
	If $i=r-1$, then both paths $x_{r-1}x_{r-2}\cdots x_1 x_{2r-2}x_{2r-1}$ and $x_{r-1}x_{r-2}\cdots x_1 x_{2r-1}x_{2r-2}$ are canceling (since $x_1\cdots x_{r}$ uses each of the first $r-1$ colors and $x_{2r-2}x_{2r-1}$ has the same color as $x_{r-1}x_r$), see Figure~\ref{fig:rPath} for the case $r=3$.  Thus in this case there is a canceling path from $x_i$ to $x_j$ regardless of whether $j=2r-2$ or $j=2r-1$.  So we can assume $i=r-2$ and hence $j=2r-2$.  In this case we can consider the path $x_{2r-2}x_{2r-1}\cdots x_{3r-4}x_{r-1}x_{r-2}$ and again this is canceling.  We conclude that the coloring $\chi$ is $(r,k)$-canceling, proving the result.
	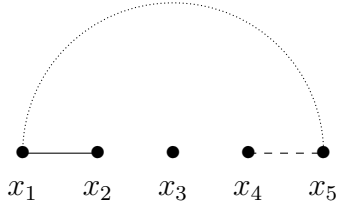
\begin{figure}
		\[
	\begin{tikzpicture}
		\node at (0,0) {$\bullet$};
		\node at (1,0) {$\bullet$};
		\node at (2,0) {$\bullet$};
		\node at (3,0) {$\bullet$};
		\node at (4,0) {$\bullet$};
		\node at (0,-.5) {$x_1$};
		\node at (1,-.5) {$x_2$};
		\node at (2,-.5) {$x_3$};
		\node at (3,-.5) {$x_4$};
		\node at (4,-.5) {$x_5$};
		
		\draw (0,0)--(1,0);
		\draw[dashed] (3,0)--(4,0);
		\draw[densely dotted] (4,0) arc [radius=2, start angle=0, end angle=180];
	\end{tikzpicture}
		\]
		\caption{The path from $x_{2}=x_{r-1}$ to $x_{4}=x_{2r-2}$ when $r=3$.  Note that $x_2x_3$ is a dashed edge, $x_3x_4$ is a solid edge, and all the other missing edges are dotted.}\label{fig:rPath}
	\end{figure}
\end{proof}

\section{Concluding Remarks}\label{sec:concluding}
There are many directions one could consider for further research.  One is to continue to explore $k$-canceling and more generally $(r,k)$-canceling graphs.  In particular, it would be nice to know when $K_n$ is $(r,k)$-canceling, since there exist $n$-vertex $(r,k)$-canceling graphs if and only if $K_n$ is $(r,k)$-canceling.

\begin{quest}
	Let $n_{r,k}$ be the smallest integer such that $K_n$ is $(r,k)$-canceling for all $n\ge n_{r,k}$.  What is $n_{r,k}$ (approximately)?  In particular, what are $n_{2,k}$ and $n_{r,1}$ (asymptotically) equal to?
\end{quest}
For example, Propositions~\ref{prop:Knk} and \ref{prop:Knr} show that $n_{r,k}\le 3kr$ for all $k,r\ge 1$.   It is not difficult to show that $n_{2,1}=4,\ n_{2,2}=5$, and $n_{2,3}=7$ by using Lemma~\ref{lem:Kn}, Proposition~\ref{prop:Knk}, and some casework for small $n$. 
Our best general bounds for $r=2$ are the following.
\begin{prop}
	For $k\ge 5$ we have
	\[k+\log_4(k)\le n_{2,k}\le 2k+4.\]
\end{prop}
\begin{proof}
	The upper bound follows from Proposition~\ref{prop:Knk}.  A well known result in Ramsey theory \cite{ES} says that any 2-coloring of $K_n$ contains a monochromatic clique on at least $\log_4(n)$ vertices (or equivalently, if $n\ge 4^t$, then any 2-coloring of $K_n$ contains a monochromatic $K_t$).  Thus if $n=k-1+\log_4(k)$, then any signing $\sig$ of $K_n$ yields a monochromatic clique on a set $T$ of size at least $\log_4(k)$.  Taking $S=V(K_n)\sm T$ (which is a set of size less than $k$) gives that $G-S$ is a monochromatic clique on at least 2 vertices, and hence $W_\sig(G-S)>0$.  Thus there exists no $k$-canceling signing of $K_n$ as desired.
\end{proof}
It is not difficult to improve this lower bound with a more careful analysis, but ultimately we do not know how to a prove a stronger bound than $k+\Om(\log k)$.  For $k=1$ and large $r$, we know essentially nothing beyond the upper bound $n_{r,1}\le 3(r-1)$ given by Proposition~\ref{prop:Knk}, though we think $n_{r,1}$ may be asymptotically smaller than this.

While this paper focused primarily on generalizing \v{S}olt{\'e}s' problem to signed graphs, there are many other classical problems in the study of Wiener indices that can be asked in this setting.  For example, a well known result for Wiener indices is that if $T$ is an $n$-vertex tree, then $W(S_n)\le W(T)\le W(P_n)$ where $S_n,P_n$ are the $n$-vertex star and path graphs, respectively.  One could ask for similar bounds in the signed setting, and we suspect the following is true.
\begin{conj}\label{conj:tree1}
	If $(T,\sig)$ is a signed $n$-vertex tree, then
	\[W_\al(P_n)\le W_\sig(T)\le W_+(P_n),\]
	where $+$ is the constant signing that assigns $+1$ to every edge of $P_n$, and $\al$ is the alternating signing which assigns the first edge of the path $+1$, the second $-1$, the third $+1$, and so on.
\end{conj}
Note that $W_+(P_n)=W(P_n)$, so the upper bound follows from the result for classical Wiener indices, and it remains to prove that the lower bound holds.

Another related statistic for a graph one could consider is the \textit{minimum signed Wiener index} $W_*(G):=\min_\sig(G)$, where the minimum ranges over all signings $\sig$ of $G$.  This statistic is analogous to the minimum digraph Wiener index of all orientations of a graph $G$ defined by Knor, Majstrovi\'c, and \v{S}krevoski~\cite{Digraph}.  Again one could ask for the extremal values of $W_*(T)$ when $T$ is an $n$-vertex tree.    To state our conjecture regarding this, we say that a tree $T$ is a \textit{double star} if there exist vertices $x,y\in V(T)$ such that every edge of $T$ uses at least one of the vertices $x$ or $y$. For example, the star $S_n$ is a double star since every edge is incident to the center of the star $x$.
\begin{conj}
	If $T$ is an $n$-vertex tree, then \[W_*(P_n)\le W_{*}(T)\le \max_{D\in \c{D}}W_*(D),\]
	where $\c{D}$ is the set of all $n$-vertex double stars.
\end{conj}
We have verified this conjecture for $n\le 9$, and we note that this conjecture is false if one only considers stars as opposed to double stars.  This conjecture is somewhat surprising to us because it is essentially the opposite of what happens for the classical case where paths maximize $W(T)$ and stars minimize $W(T)$.

We end this paper by considering Dyck paths, which are one of the most well studied objects in combinatorics.  These are lattice paths from $(0,0)$ to $(2n,0)$ which use $n$ steps along the vector $(1,1)$ and $n$ steps along the vector $(1,-1)$ and such that the path always stays above the $x$-axis.  It is natural to view a Dyck path as a signed path graph by assigning the $(1,x)$ steps the sign $x$, see Figure~\ref{fig:Dyck}.  This leads to the following (somewhat vague) question.

\begin{figure}
	\[
	\begin{tikzpicture}[scale=.7]
		\node at (0,0) {$\bullet$};
		\node at (1,1) {$\bullet$};
		\node at (2,0) {$\bullet$};
		\node at (3,1) {$\bullet$};
		\node at (4,2) {$\bullet$};
		\node at (5,1) {$\bullet$};
		\node at (6,2) {$\bullet$};
		\node at (7,1) {$\bullet$};
		\node at (8,0) {$\bullet$};
		
		\draw (0,0)--(1,1);
		\draw[dashed] (1,1)--(2,0);
		\draw (2,0)--(3,1);
		\draw (3,1)--(4,2);
		\draw[dashed] (4,2)--(5,1);
		\draw (5,1)--(6,2);
		\draw[dashed] (6,2)--(7,1);
		\draw[dashed] (7,1)--(8,0);
	\end{tikzpicture}
	\]
	
	\caption{A Dyck path viewed as a sign graph.}\label{fig:Dyck}
\end{figure}
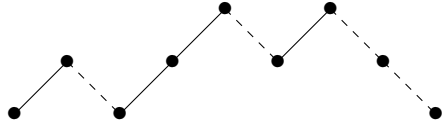

\begin{quest}
	What can be said about the signed Wiener indices of Dyck paths?
\end{quest}
Possible ways to answer this question might be to count the number of paths with a given signed Wiener index, or to determine which signed Wiener indices are achievable by Dyck paths.
\bibliographystyle{abbrv}
\bibliography{Wiener}

\begin{thebibliography}{10}

\bibitem{DEG}
A.~A. Dobrynin, R.~Entringer, and I.~Gutman.
\newblock Wiener index of trees: theory and applications.
\newblock {\em Acta Applicandae Mathematica}, 66(3):211--249, 2001.

\bibitem{ES}
P.~Erd{\"o}s and G.~Szekeres.
\newblock A combinatorial problem in geometry.
\newblock {\em Compositio mathematica}, 2:463--470, 1935.

\bibitem{FK}
G.~Fan and H.~A. Kierstead.
\newblock Hamiltonian square-paths.
\newblock {\em journal of combinatorial theory, Series B}, 67(2):167--182,
  1996.

\bibitem{KMS}
M.~Knor, S.~Majstorovic, and R.~{\v{S}}krekovski.
\newblock Some results on wiener index of a graph: an overview.
\newblock {\em 2 nd Croatian Combinatorial Days}, page~49.

\bibitem{KMS3}
M.~Knor, S.~Majstorovi{\'c}, and R.~{\v{S}}krekovski.
\newblock Graphs preserving wiener index upon vertex removal.
\newblock {\em Applied Mathematics and Computation}, 338:25--32, 2018.

\bibitem{KMS2}
M.~Knor, S.~Majstorovi{\'c}, and R.~{\v{S}}krekovski.
\newblock Graphs whose wiener index does not change when a specific vertex is
  removed.
\newblock {\em Discrete Applied Mathematics}, 238:126--132, 2018.

\bibitem{KST}
M.~Knor, R.~{\v{S}}krekovski, and A.~Tepeh.
\newblock Mathematical aspects of wiener index.
\newblock {\em Ars Mathematica Contemporanea}, 11(2):327--352, 2016.

\bibitem{Digraph}
M.~Knor, R.~{\v{S}}krekovski, and A.~Tepeh.
\newblock Some remarks on wiener index of oriented graphs.
\newblock {\em Applied Mathematics and Computation}, 273:631--636, 2016.

\bibitem{KSS}
J.~Koml{\'o}s, G.~N. S{\'a}rk{\"o}zy, and E.~Szemer{\'e}di.
\newblock On the square of a hamiltonian cycle in dense graphs.
\newblock {\em Random Structures \& Algorithms}, 9(1-2):193--211, 1996.

\bibitem{S}
L.~{\v{S}}olt{\'e}s.
\newblock Transmission in graphs: a bound and vertex removing.
\newblock {\em Mathematica Slovaca}, 41(1):11--16, 1991.

\bibitem{T}
N.~Trinajstic.
\newblock {\em Chemical graph theory}.
\newblock Routledge, 2018.

\bibitem{W}
H.~Wiener.
\newblock Structural determination of paraffin boiling points.
\newblock {\em Journal of the American chemical society}, 69(1):17--20, 1947.

\bibitem{xu2014survey}
K.~Xu, M.~Liu, K.~C. Das, I.~Gutman, and B.~Furtula.
\newblock A survey on graphs extremal with respect to distance-based
  topological indices.
\newblock {\em MATCH Commun. Math. Comput. Chem}, 71(3):461--508, 2014.

\end{thebibliography}

\end{document}